\documentclass[12pt,a4paper]{amsart}
\usepackage{graphicx,multirow,array,amsmath,amssymb,kotex}

\usepackage{xcolor}
\usepackage{longtable}
\usepackage{supertabular}
\usepackage{cite}

\newtheorem{theorem}{Theorem}

\newtheorem{definition}[theorem]{Definition}

\begin{document}
	
\title{Ribbonlength of twisted torus knots}
	
\author[H. Kim]{Hyoungjun Kim}
\address{College of General Education, Kookmin University, Seoul 02707, Korea}
\email{kimhjun@kookmin.ac.kr}
\author[S. No]{Sungjong No}
\address{Department of Mathematics, Kyonggi University, Suwon 16227, Korea}
\email{sungjongno@kgu.ac.kr}
\author[H. Yoo]{Hyungkee Yoo}
\address{Research Institute for Natural Science, Hanyang University, Seoul 04763, Korea}
\email{hyungkee@hanyang.ac.kr}

\thanks{The first author(Hyoungjun Kim) was supported by the National Research Foundation of Korea (NRF) grant funded by the Korea government Ministry of Science and ICT(NRF-2021R1C1C1012299).}
\thanks{The second author(Sungjong No) was supported by the National Research Foundation of Korea(NRF) grant funded by the Korea government Ministry of Science and ICT(NRF-2020R1G1A1A01101724).}

\begin{abstract}
The ribbonlength Rib$(K)$ of a knot $K$ is the infimum of the ratio of the length of any flat knotted ribbon with core $K$ to its width.
A twisted torus knot $T_{p,q;r,s}$ is obtained from the torus knot $T_{p,q}$ by twisting $r$ adjacent strands $s$ full twists.
In this paper, we show that the ribbonlength of $T_{p,q;r,s}$ is less then or equal to $2(\max \{ p, q, r \} +|s|r)$ where $p$ and $q$ are positive.
Furthermore, if $r \leq p-q$, then the ribbonlength of $T_{p,q;r,s}$ is less then or equal to $2(p+(|s|-1)r)$.
\end{abstract}
	
\maketitle

\section{Introduction} \label{sec:intro}
Knot theory began with the work of physicist Tait on the knotted vortex in the ether, which is supposed to be a medium for electromagnetic waves.
Since then,
numerous mathematicians have investigated the geometric and topological properties of knots,
leading to the development of knot theory and manifold theory.
After that, the physical properties of the knot were discussed again,
and a study was conducted on the ropelength as it relates to the energy of the knot~\cite{BS,BS2,FHW,O1,O2,O3}.
The {\em ropelength\/} is the infimum ratio of the length of given knot to its thickness.
Kauffman~\cite{K} suggests the ribbonlength which is the two-dimensional version of ropelength.

\begin{figure}[h!]
	\includegraphics{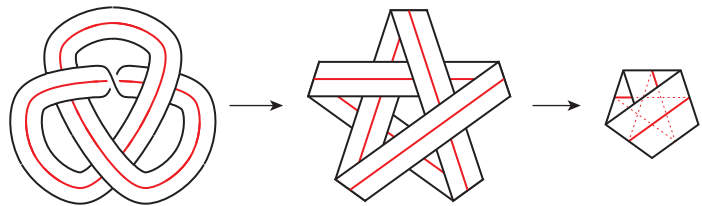}
	\caption{A knotted ribbon}
	\label{fig:11}
\end{figure}

The following definitions were introduced in~\cite{DHLM,DKTZ}.

\begin{definition}
Given an oriented polygonal knot (or link) diagram $K$, we say the folded
ribbon knot $K_w$ of width $w$ is allowed provided
	\begin{enumerate}
		\item The ribbon has no singularities (is immersed), except at the fold lines which are assumed
		to be disjoint.
		\item $K_w$ has consistent crossing information, and moreover this agrees
		\begin{enumerate}
			\item with the folding information at each fold, and
			\item with the crossing information of the knot diagram $K$.
		\end{enumerate}
	\end{enumerate}
\end{definition}

\begin{definition}
Take a knot (or link) $K$ and a positive real number $w$.
Then
$$\emph{\text{Rib}}(K)=\inf_{K' \in [K]_w}\frac{\emph{\text{Len}}(K'_w)}{w}$$
is called a ribbonlength of $K$
where $[K]_w$ is the set of knots (or links) that are equivalent to $K$ and have a folded ribbon knot $K'_w$.
\end{definition}

In ~\cite{K}, Kauffman showed that each upper bound of a ribbonlength of the truncated trefoil knot and the truncated figure eight knot are $\frac{4}{\sqrt{5-2\sqrt{5}}}$ and $\frac{32}{\sqrt{15}}$ respectively.
He used the pentagonal and the hexagonal folded ribbon knots.
For torus knots,
several researchers have given upper bounds of ribbonlength ~\cite{DHLM, KMRT}

In this paper, we deal with the ribbonlength of a specific type of knot called a twisted torus knot.
A {\em twisted torus knot\/} $T_{p,q;r,s}$ is obtained from the torus knot $T_{p,q}$
by twisting $r$ adjacent strands $s$ full twists.
Since the ribbonlength of any knot is the same value of its mirror image, it is sufficient to consider the case that $p$ and $q$ are positive.

\begin{theorem} \label{thm:twisted1}
For any positive integers $p,q,r$ with $r<p+q$ and nonzero integer $s$,
$$\emph{\text{Rib}}(T_{p,q;r,s}) \leq 2(\max \{ p, q, r \} +|s|r).$$
\end{theorem}

Especially, when $r$ is less than or equal to $p-q$, we obtain sharper upper bounds.

\begin{theorem} \label{thm:twisted2}
For any positive integers $p,q,r$ with $r \leq p-q$ and nonzero interger $s$,
$$\emph{\text{Rib}}(T_{p,q;r,s}) \leq 2(p+(|s|-1)r).$$
\end{theorem}

If $r$ is equal to 1, then $T_{p,q;r,s}$ has the same knot type with $T_{p,q}$ regardless of the value of $s$.
This implies that when we take the value of $s$ as 1 in Theorem~\ref{thm:twisted2}, we have the result that $\text{Rib}(T_{p,q}) \leq 2p$.
It is also found by Denne el al~\cite{DHLM}.

In Section~\ref{sec:preliminary}, we introduce a twisted torus knot and its related results.
Moreover, we give some notations and the folding rules which is needed to construct the folded ribbon knot.
In Section~\ref{sec:proof},
we prove Theorem~\ref{thm:twisted1} and \ref{thm:twisted2}.

\section{Preliminary} \label{sec:preliminary}
In this section we introduce some basic concepts for our work.

\subsection{Twisted torus knots} \label{subsec:twisted}\ 

In~\cite{Dea}, Dean introduced the twisted torus knots for a study of Seifert-fibered surgery on knots.
For any relatively prime integers $p$ and $q$,
we can regard a torus knot $T_{p,q}$ to a simple closed curve on the genus one Heegaard surface of $S^3$ with slope $\frac{p}{q}$.
For an integer $r$ between 2 and $p+q$,
take an unknotted circle $C$ in the exterior of $T_{p,q}$ wrapping around $r$ adjacent strands of $T_{p,q}$.
For any non-zero integer $s$,
$S^3$ is obtained again by performing $-\frac{1}{s}$ surgery on $C$ in $S^3$.
However the surgery transforms $T_{p,q}$ into a new knot on the genus two Heegaard surface of $S^3$ as drawn in Figure~\ref{fig:21}.
This knot is called a twisted torus knot, and is denoted by $T_{p,q;r,s}$.

\begin{figure}[h!]
	\includegraphics{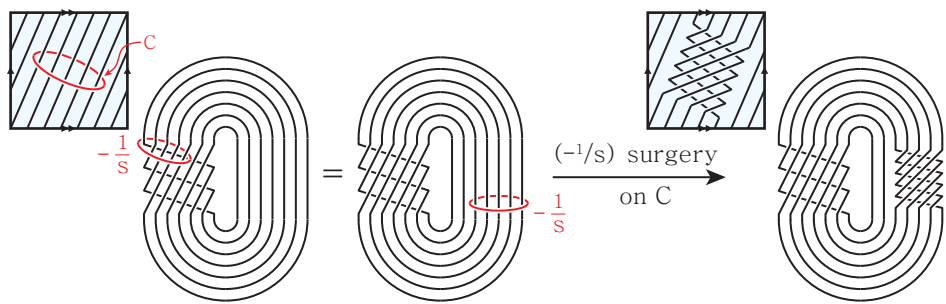}
	\caption{A twisted torus knot}
	\label{fig:21}
\end{figure}

We remark that when $r$ is equal to 1, $T_{p,q;r,s}$ has the same knot type as $T_{p,q}$ regardless of the value of $s$.
So $r$ is considered to be greater than or equal to 2 in the former studies for twisted torus knots.
For a knotted ribbon, when $r$ is equal to 1, $s$ affects the framing of the knot, which affects the twist of the ribbon.
Thus we include the case that $r$ is equal to 1.

Note that $T_{p,q;r,s}$ is a mirror image of $T_{p,-q;r,-s}$.
Since any knot has the same ribbonlength with its mirror image, we only consider the case that $p$ and $q$ are positive.
In~\cite{L1}, Lee showed that $T_{p,q;r,s}$ has the same knot type as $T_{q,p;r,s}$.
Thus we assume that $2 \leq q < p$.
Note that if $r$ is equal to $p$ or multiple of $q$, then $T_{p,q;r,s}$ is one of a cable knot, a torus knot, or a trivial knot.
In addition, Lee~\cite{L1,L2,L3,L4,L5,L6} discovered the conditions for $T_{p,q;r,s}$ to be either a satellite knot, a torus knot, or a trivial knot.

\begin{figure}[h!]
	\includegraphics{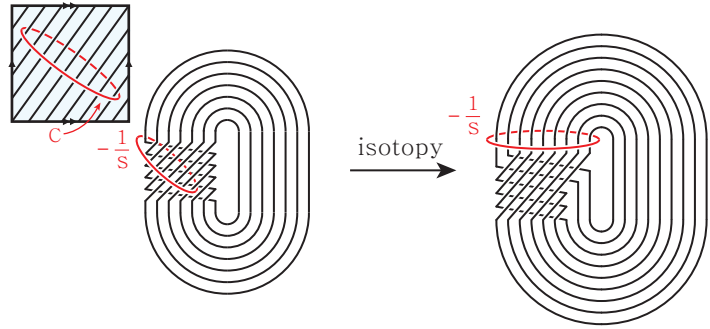}
	\caption{A twisted torus knot $T_{p,q;r,s}$ when $r>p$}
	\label{fig:22}
\end{figure}

If $r$ is less than or equal to $p$,
$T_{p,q;r,s}$ is $s$ full twist of $r$ adjacent strands of $T_{p,q}$.
When $r$ is greater than $p$,
we need the transformation shown in ~\cite{L3}.
Take the $r-p$ adjacent extra strands.
After kinking it, we obtain the braid with $r$ strands as drawn in Figure~\ref{fig:22}.
Then $T_{p,q;r,s}$ is obtained by $-\frac{1}{s}$ surgery on $C$.

\subsection{Weighted bands and folding rules} \label{subsec:rule}\

As shown in Figures~\ref{fig:21} and \ref{fig:22}, a twisted torus knot can be considered as the union of parallel strands.
So we define weighted strands and boxes as following.
A {\it weighted $n$ strand} is a bunch of $n$ parallel strands as drawn in Figure~\ref{fig:23} (a).
A single positive kinked weighted strand with $n$ is $n$ parallel strands with one full twist as drawn in Figure~\ref{fig:23} (b).
Note that a single negative kinked weighted strand means $-1$ full twist.
When a weighted strand with $n$ has $s$ full twists,
we represent this by attaching a {\it box with $s$} to the weighted strand as drawn in Figure~\ref{fig:23} (c).

\begin{figure}[h!]
	\includegraphics{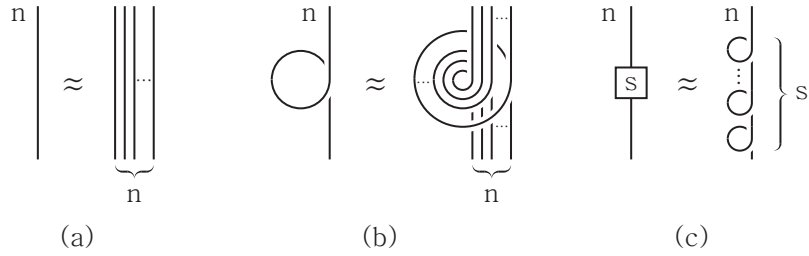}
	\caption{Weighted strands and a box on a weighted strand}
	\label{fig:23}
\end{figure}

Similar with the strands, we also use the weighted band.
A {\it weighted band} with $n$ is a bunch which consists of $n$ parallel bands as drawn in Figure~\ref{fig:24} (a).
Folding a weighted band with $n$ corresponds to  folding $n$ parallel bands at the same time as drawn in Figure~\ref{fig:24} (b).

\begin{figure}[h!]
	\includegraphics{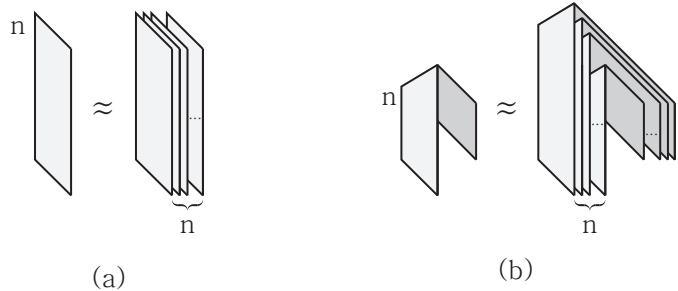}
	\caption{Weighted bands}
	\label{fig:24}
\end{figure}

Now we give some folding types to construct a folded ribbon knot.
Folding types 1, 2 and 3 are represented in Figure~\ref{fig:25}.
We usually use the folding type 1 for a non-twisted band.
However, if there are two non-twisted bands which are not parallel, then the overlapping part occurs when both of them use the folding type 1.
Therefore, in this case, we use the folding type 1 for one band, and the folding type 2 for the other.
Since the folding type 3 in the figure has one full twist, we use this folding type for a positively twisted band.
For a negatively twisted band, we fold the opposite direction for each folding.
The length of each of three folding types is equal to $2w$ where $w$ is the width of the ribbon.

\begin{figure}[h!]
	\includegraphics{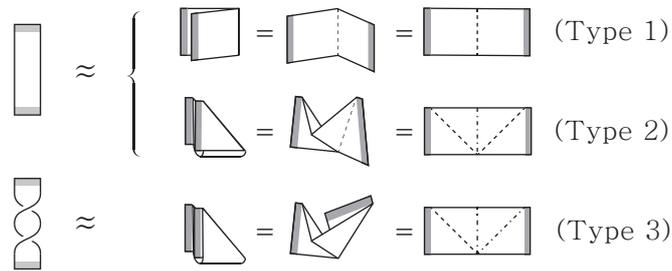}
	\caption{Folding types 1, 2 and 3}
	\label{fig:25}
\end{figure}

Figure~\ref{fig:26} shows the folding type 4 for a positive $s$.
This type is used for a box with $s$ on a weighted band with width $w$.
We remark that if we take the opposite direction of each folding in the process of the folding type 4 with $s$, then it is equal to the folding type 4 with $-s$.
So it is sufficient to consider the case that $s$ is positive.
Note that except the rolled part, it is the same with the folding type 3.
Thus the rolled part is corresponding to $s-1$ full twists.
Since rolling a band once corresponds to one full twist, the ribbonlength of the folding type 4 for a box with $s$ on the weighted band is equal to $2ws$ where $w$ is the width of the ribbon.

\begin{figure}[h!]
	\includegraphics{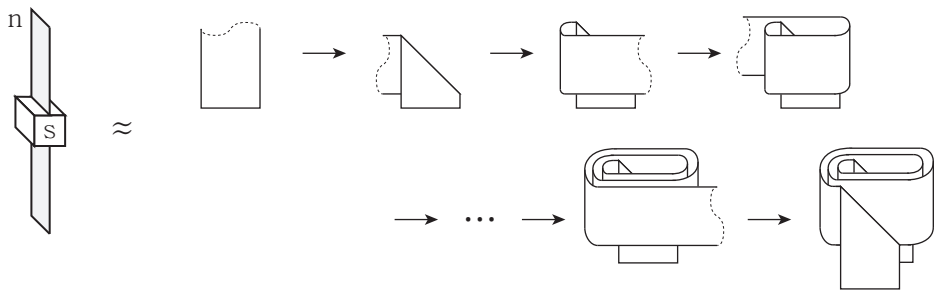}
	\caption{A box on the weighted band and the folding type 4}
	\label{fig:26}
\end{figure}

In the next section, we build a folded ribbon knot by combining bands with the four folding types.

\section{Proof of Theorems} \label{sec:proof}

In this section, we prove Theorem~\ref{thm:twisted1} and Theorem~\ref{thm:twisted2} using the observations in Section~\ref{sec:preliminary}.
As we mentioned in the previous section, assume that $2 \leq q < p$.
The braid word of $T_{p,q}$ is the form $(\sigma_1 \sigma_2 ... \sigma_{p-1})^q$.
By the braid relations,
$$
(\sigma_1 \sigma_2 ... \sigma_{p-1})^q
=(\sigma_1 \sigma_2 ... \sigma_{q-1})^q(\sigma_q \sigma_{q+1} ... \sigma_{p-1})(\sigma_{q-1} \sigma_{q} ... \sigma_{p-2})...(\sigma_1 \sigma_2 ... \sigma_{p-q}).
$$
Thus first $q$ strings are fully twisted, and the other $p-q$ strings are untwisted.
Using this fact,
we divide the braid of $T_{p,q;r,s}$ to bunches of fully twisted parallel strands.

\begin{proof}[Proof of Theorem~\ref{thm:twisted1}]
First assume that $p<r<p+q$.
Then the braid of $T_{p,q;r,s}$ is represented as Figure~\ref{fig:31} (a).
This braid is divided into the upper and lower braid such that the upper braid $B_U$ consists of $r$ strands with $s$ full twists, and the lower braid $B_L$ is the remaining part of the braid.
Note that $B_L$ is divided into three bunches, each bunch consists of parallel strands as drawn in Figure~\ref{fig:31} (b).
So the braid of $T_{p,q;r,s}$ is represented by weighted strands as drawn in Figure~\ref{fig:31} (c).
The closure of this braid is obtained by identifying the top horizontal line of $B_U$ and the bottom horizontal line of $B_L$ as drawn in Figure~\ref{fig:31} (d).

\begin{figure}[h!]
	\includegraphics{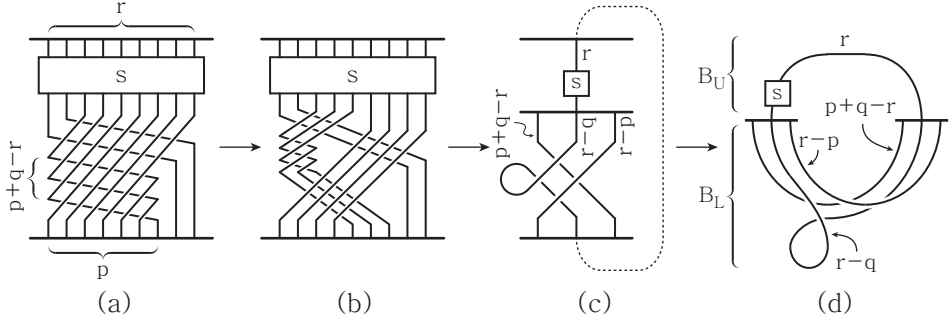}
	\caption{Weighted strands when $p < r < p+q$}
	\label{fig:31}
\end{figure}

To obtain the folded ribbon knot of $T_{p,q;r,s}$ with width $w$, we replace its all strands into bands.
So the weighted strands which consists of parallel strands are replaced into weighted bands which consists of parallel bands as drawn in Figure~\ref{fig:32} (a).
Now replace each weighted band into a ribbon which is folded in a suitable way as drawn in Figure~\ref{fig:32} (b).
In the braid, the weighted bands with $r-p$, $p+q-r$, $r-q$ and $r$ are folded into the folding type 1, 2, 3 and 4, respectively as drawn in Figure~\ref{fig:32} (c).
Two weighted bands with $p+q-r$ and $r-q$ are folded into the folding type 2 and 3, respectively, and they are put as drawn in Figure~\ref{fig:32} (b).
After replacing, we shrink the bands between $B_U$ and $B_L$.

\begin{figure}[h!]
	\includegraphics{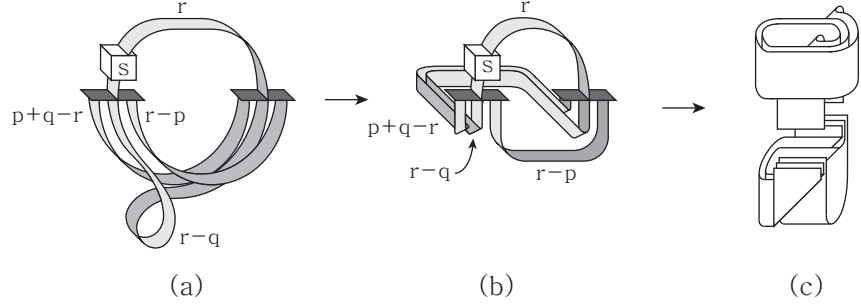}
	\caption{Weighted bands and a folded ribbon knot when $p<r<p+q$}
	\label{fig:32}
\end{figure}

Now we calculate the total length of this ribbon.
Note that each band with the folding type 1, 2 and 3 needs length $2w$, and a band with type 4 needs length $2w|s|$.
Since each portion of the ribbon in $B_L$ has the folding types 1, 2 or 3, the sum of length of them is equal to $2wr$.
Since each portion of the ribbon in $B_U$ has the folding types 4, the sum of length of them is equal to $2w|s|r$.
Therefore, the total length of this ribbon is equal to $2w(r+|s|r)$ where $w$ is the width of the ribbon.

Next we assume that $r \leq p$.
Then the braid of $T_{p,q;r,s}$ is represented as Figure~\ref{fig:33} (a).
Similar with the previous case, this braid is divided into the upper and lower braid $B_U$ and $B_L$, and each of them is divided into two bunches.
So the braid of $T_{p,q;r,s}$ is represented by weighted strands as drawn in Figure~\ref{fig:33} (b).
The closure of this braid is represented in Figure~\ref{fig:33} (c).
In the braid, the weighted bands with $p-q$, $q$ and $r$ are folded into the folding type 1, 3 and 4, respectively, and the remaining $p-r$ bands are directly connected as drawn in Figure~\ref{fig:33} (d).
After that, we shrink the bands between $B_U$ and $B_L$.

\begin{figure}[h!]
	\includegraphics{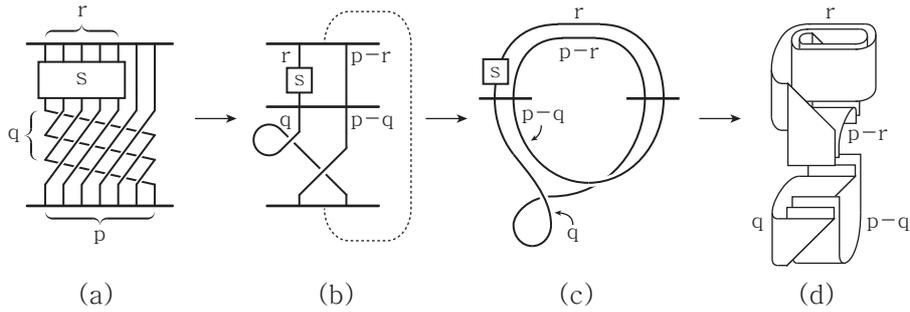}
	\caption{Weighted strands and a folded ribbon knot when $r \leq p$}
	\label{fig:33}
\end{figure}

Now we calculate the total length of this ribbon.
Since each portion of the ribbon in $B_L$ has the folding types 1 or 3, the sum of length of them is equal to $2wp$.
In $B_U$, $r$ portions have the folding type 4 and the others are directly connected.
Thus the sum of length of them is equal to $2w|s|r$.
Therefore, the total length of this ribbon is equal to $2w(p+|s|r)$ where $w$ is the width of the ribbon.
\end{proof}

If $r$ is less than or equal to $p-q$,
then we can reduce more length by modifying the braid of $T_{p,q;r,s}$.
Using this fact, we prove Theorem~\ref{thm:twisted2}.

\begin{proof}[Proof of Theorem~\ref{thm:twisted2}]
We assume that $r \leq p-q$.
Then the braid of $T_{p,q;r,s}$ is represented as Figure~\ref{fig:34} (a).
As in the figure, move the $s$-full twists along the closure of the braid to the top of the braid through the torus knot $T_{p,q}$.
In this way, it is shifted to the right by $q$ strands while passing the braid of the torus knot $T_{p,q}$.
We can take three weighted strands as drawn in Figure~\ref{fig:34} (b).
The closure of this braid is represented in Figure~\ref{fig:34} (c).

\begin{figure}[h!]
	\includegraphics{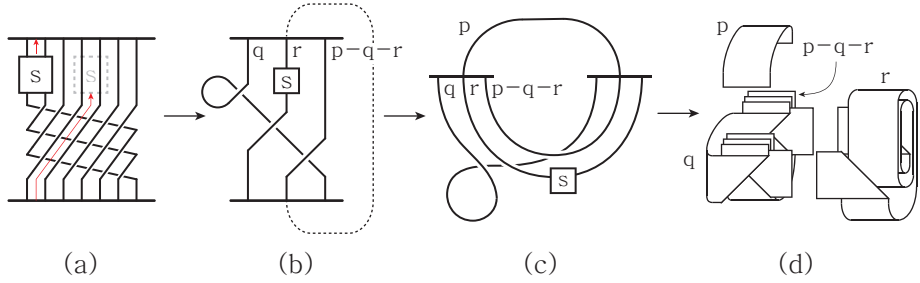}
	\caption{Weighted strands and a folded ribbon knot when $r \leq p-q$}
	\label{fig:34}
\end{figure}

First, the weighted bands with $r$ is folded into a combined form of the folding type 3 and 4.
In this process, since the folding type 3 makes one full-twist, we make the part of the folding type 4 have $s-1$ full twists when $s$ is positive.
If $s$ is negative, we use the mirror image.
Next, the weighted bands with $p-q-r$ and $q$ are folded into the folding type 1 and 3, respectively, and the remaining $p$ bands are directly connected as drawn in Figure~\ref{fig:34} (d).
After that, we shrink the extra bands.

Now we calculate the total length of this ribbon.
Note that the portion of the ribbon for the weighted band with $r$ is folded into a combined form of the folding type 3 and 4, and this portion needs length $2w|s|$.
So the sum of length of them is equal to $2w|s|r$.
Except for the portion that is directly connected, the remaining $p-r$ portions have the folding type 1 or 3.
Thus the sum of length of them is equal to $2w(p-r)$.
Therefore, the total length of this ribbon is equal to $2w(p+(|s|-1)r)$ where $w$ is the width of the ribbon.
\end{proof}

\bibliography{ribtt.bib} 
\bibliographystyle{abbrv}
	
\end{document}